\documentclass{scrartcl}

\usepackage{amssymb,amsmath,amsthm}
\usepackage{xcolor}
\usepackage[hidelinks,hypertexnames=true,
  colorlinks,
  linkcolor={red!50!black},
  citecolor={blue!50!black},
  urlcolor={blue!80!black}]{hyperref}

\newtheorem{theorem}{Theorem}

\def\bu{\boldsymbol u}

\def\bf{\boldsymbol f}

\def\bV{\boldsymbol V}
\def\norm#1{\|#1\|}
\def\nR{\mathbb R}
\def\nN{\mathbb N}
\def\beqa{\begin{eqnarray*}}
\def\eeqa{\end{eqnarray*}}
\def\beqal{\begin{eqnarray}}
\def\eeqal{\end{eqnarray}}
\def\div{\mbox{div}\,}

\def\nR{\mathbb{R}}
\def\vn{\vec n}
\def\vt{\vec t}
\renewcommand{\div}{\operatorname{div}}

\begin{document}

\title{Inf-sup condition for Stokes with outflow condition}
\author{M. Braack\thanks{%
    Kiel University, Department of Mathematics, Kiel, Germany. \url{braack@math.uni-kiel.de}
  }
  \and T. Richter\thanks{%
    Otto von Guericke University Magdeburg, Institute for Analysis and Numerics, Magdeburg, Germany. \url{thomas.richter@ovgu.de}
  }}

\maketitle

\begin{abstract}
    The inf-sup condition is one of the essential tools in the analysis of the Stokes equations and especially in numerical analysis. In its usual form, the condition states that for every pressure $p\in L^2(\Omega)\setminus\mathbb{R}$ (i.e. with mean value zero) a velocity $\bu\in H^1_0(\Omega)^d$ can be found, so that $(\operatorname{div}\,\bu,p)=\|p\|^2$ and $\|\nabla \bu\|\le c \|p\|$ applies, where $c>0$ does not depend on $\bu$ and $p$. 

However, if we consider domains that have a Neumann-type outflow condition on part of the boundary $\Gamma_N\subset\partial\Omega$, the inf-sup condition cannot be used in this form, since the pressure here comes from $L^2(\Omega)$ and does not necessarily have zero mean value. 
In this note, we derive the inf-sup condition for the case of outflow boundaries.  
\end{abstract}


\section{Problem statement and notation}
We consider the Stokes equation
\begin{equation}\label{stokes}
-\Delta \bu + \nabla p = \bf,\quad \div\,\bu=0
\end{equation}
on a domain $\Omega\subset\mathbb{R}^d$, $d\in\nN$. We assume that the boundary of $\Omega$ consists of two parts, $\partial\Omega = \Gamma_D\cup \Gamma_N$. 
Both parts are assumed to have positive measure $|\Gamma_D|>0$, $|\Gamma_N|>0$.
On $\Gamma_D$ a Dirichlet condition and on $\Gamma_N$ a the Neumann condition
\begin{equation}\label{donothing}
\partial_{\vn} \bu -p\vn = 0\quad\mbox{on }\Gamma_N
\end{equation}
holds, where $\vn$ is the outward facing normal vector. This condition is called the \emph{do-nothing} condition and it is the natural condition for the variational formulation of the Stokes problem~\cite{HeywoodRannacherTurek1992}. 

The solution to~\eqref{stokes},~\eqref{donothing} is found in the spaces $\bu\in H^1_0(\Omega;\Gamma_D)^d$ and $p\in Q:=L^2(\Omega)$, where $H^1(\Omega)$ is the space of $L^2(\Omega)$ functions with first weak derivative in $L^2(\Omega)$. By $\bV:=H^1(\Omega;\Gamma_D)^d=\{ \bu\in H^1(\Omega)^d : \bu|_{\Gamma_D}=0\}$ we denote the vector-valued $H^1$ functions with 
vanishing  trace  on $\Gamma_D$. For the right hand side we allow 
linear functionals $\bf\in \bV'$.

In general, the pressure is not zero in average if an outflow boundary exists. On certain domains, e.g. if $d=2$ and the outflow boundary $\Gamma_N$ is  a line segment, and regular pressure, it holds (see~\cite{Rannacher_2017})
\begin{equation}\label{pres-reg}
\int_{\Gamma_N}p \,\text{d}s = 0.
\end{equation}
It is well-known that the divergence operator satisfies the inf-sup condition for the space $H^1_0(\Omega)^d\times L^2_0(\Omega)$ with $L^2_0(\Omega) = L^2(\Omega)\setminus\mathbb{R}$, see~\cite{Galdi}
Although flow problems with outflow conditions are of greatest importance in applications and the do-nothing condition is frequently used, the inf-sup condition for the spaces $\bV\times L^2(\Omega)$, is, to the best knowledge of the author, not found in literature. 

In the following,
\[
\|f\|_X:=\left(\int_X |f(x)|^2\,\text{d}x\right)^\frac{1}{2}
\]
denotes the $L^2$-norm and we use the short notation $\|\cdot\|:=\|\cdot\|_\Omega$ if the domain $\Omega$ is meant. By
\[
(f,g)_X = \int_X f(x)g(x)\,dx
\]
we denote the inner product and also use the short notation $(\cdot,\cdot):=(\cdot,\cdot)_\Omega$ on the domain $\Omega$.

\section{Main result}\label{sec:1}
\begin{theorem}
Let $\Omega\subset\nR^d$, $d\in\mathbb N$, a Lipschitz domain with boundary $\partial\Omega=\Gamma_D\cap\Gamma_N$.
The divergence operator satisfies the
inf-sup condition for this pair $\bV\times Q$, i.e.
\beqa
   \inf_{p\in Q\setminus\{0\}} \sup_{\bu\in \bV\setminus\{0\}}\frac{(\div \bu,p)}{\norm{\bu}_{\bV}\norm{p}} &=:& \gamma >0
\eeqa
\end{theorem}
\begin{proof}
Let $p\in Q\setminus\{0\}$ be given.
We split the pressure in the form $p = \bar p + p_0$ where  $\bar p:=|\Omega|^{-1}\int_\Omega p$, 
and $p_0=p-\bar p\in L^2_0(\Omega):=L^2(\Omega)\setminus\nR$.
Due to the  inf-sup condition for the space $H^1_0(\Omega)^d\times L^2_0(\Omega)$,
we find $\bu_0\in H^1_0(\Omega)^d$ s.t
\beqa
    (\mbox{div }\bu_0,p_0) & \ge& \gamma_0 \norm{\nabla \bu_0}_{\bV}\norm{p_0}_Q
\eeqa
for a constant $\gamma_0>0$ which does not depend on $\bu_0$ and $p_0$. 

Next, we develop the inf-sup condition for $\bar p\in \nR$. For $x_0\in\Gamma_N$ and $r>0$ let $\vn(x_0)$ the outer normal and
\[
B_r(x_0):=\{x\in \nR^d,\; |x-x_0|<r\},\quad H(x_0,r)  := B_r(x_0)\cap \Omega ,
\]
such that $H(x_0,r)=B_r(x_0)\cap \Omega$ and $B_r(x_0)\setminus H(x_0,r)$ are two connected sets, one completely within the domain and one on the outside.
Since Lipschitz domains satisfy inner cone conditions, it exists $x_0\in\Gamma_N$ and $r>0$ s.t. 
\beqa
  H&:=&H(x_0,r)\subset\Omega.
\eeqa
We denote the normal and tangential vector at $x_0$ by $\vn$ and $\vt$, respectively.
Let $\chi_{H}$ be the characteristic function w.r.t. $H$. We extend the normal on $\Gamma_N$ to
points $x\in\Omega$ in he interior by setting
\[
\vn(x) := \vn(\hat x),\; \hat x:=\arg\min \{|\hat y-x|: \hat y\in  \Gamma_N\cap H\},
\]
and define
\[
\psi(x) := r-|x-x_0|,\quad \bar\bu(x) := \chi_H(x)\psi(x) \vn(x).
\]
$\bar\bu(x)$ vanishes outside $H$, is $C^\infty$ in $H$, and is continuous by construction. Hence, $\bar\bu\in \bV$.
We obtain by the Gauss theorem and the fact that $\bar\bu|_{\partial H\setminus\Gamma_N}=0$:
\beqa
     \int_\Omega\mbox{div }\bar\bu\,dx &=&
     \int_H\mbox{div }\bar\bu\,dx = 
     \int_{\partial H}\bar\bu\cdot \vn \,ds = \int_{\partial H\cap\Gamma_N}\bar\bu\cdot \vn\,ds
     = \int_{\partial H\cap\Gamma_N}\psi\,ds
     \\
     &=&  \int_{\partial H\cap\Gamma_N} (r-|x-x_0|) \,ds
     =:c_1>0,
\eeqa
where $c_1=c_1(r,\Gamma_N)\approx r^d$ depends on the shape of the boundary segment $\Gamma_N\cap H$. 
Furthermore,
\begin{equation}\label{est:barbu}
\|\nabla \bar\bu\|_\Omega = \|\nabla\bar  \bu\|_H \le \|\vn\|_{W^{1,\infty}(H)} \|\psi\|_{H^1(H)}
=:c_2
\end{equation}
where $c_2=c_2(r,\Gamma_N)\approx r^\frac{d}{2}$ depends on the extension of the normals from the boundary segment $\Gamma_N$ into the domain.

For $\bar \bu\in \bV$, $\bu_0\in\bV\cap H^1_0(\Omega)^d$ from above and $p=\bar p+p_0\in Q$ it holds with $\|\div\,\bu_0\|\le d \|\nabla\bu_0\|$ and~\eqref{est:barbu}
\[
\begin{aligned}
(\div\,\bar\bu,p) &= (\div\,\bar\bu,\bar p) + (\div\,\bar\bu,p_0)
\ge \frac{c_1}{|\Omega|^\frac{1}{2}} \|\bar p\|_Q - dc_2 \|p_0\|_Q,\\
(\div\,\bu_0,p) &= \underbrace{(\div\,\bu_0,\bar p)}_{=0} + (\div\,\bu_0,p_0)
\ge \gamma_0 \|\bu_0\|_{\bV} \|p_0\|_Q.
\end{aligned}
\]
For $\bu\in \bV$ defined as 
\begin{equation}\label{bu}
\bu = \bu_0 + \frac{\gamma_0}{2dc_2} \|\bu_0\|_{\bV} \bar\bu,
\end{equation}
it holds
\begin{align}\nonumber
(\div\,\bu,p) &\ge \frac{\gamma_0}{2} \|\bu_0\|_{\bV} \Big( \|p_0\|_Q + \frac{c_1}{dc_2|\Omega|^{\frac{1}{2}} }\|\bar p\|_Q\Big) \\
&\ge\frac{\gamma_0}{2}\min\left\{1,\frac{c_1}{dc_2|\Omega|^{\frac{1}{2}}}\right\} \|\bu_0\|_{\bV} \|p\|_Q.
\label{infsup0}
\end{align}
Finally, $\bu$ defined in~\eqref{bu}, can be bounded with~\eqref{est:barbu}
\[
\|\bu\|_{\bV} \le \|\bu_0\|_{\bV} + \frac{\gamma_0}{2d}\|\bu_0\|_{\bV}
=\Big(1+\frac{\gamma_0}{2d}\Big) \|\bu_0\|_{\bV}
\]
to further estimate~\eqref{infsup0}
\[
(\div\,\bu,p) \ge  \gamma \|\bu\|_{\bV} \|p\|_Q,
\]
with
\[
\gamma :=
\frac{\gamma_0}{2}\min\left\{1,\frac{c_1}{dc_2|\Omega|^{\frac{1}{2}}}\right\}
\Big(1+\frac{\gamma_0}{2d}\Big)^{-1}.
\]
\end{proof}

\section{Discussion and outlook}

The exact value of the inf-sup constant $\gamma_0$ is only known for few domains, it can, however, be approximated numerically~\cite{Gallistl}. For anisotropic rectangular domains $\Omega:=(0,L)\times (0,H)$  with $L>H$ it is known, see~\cite{infsupaniso}, that $\gamma_0$ scales as
\[
\gamma_0 = {\mathcal O}(H/L).
\]
The constant $\gamma$ from the theory mainly shows the scaling of $\gamma_0$. Besides it depends on the values $c_1$ and $c_2$ that mainly depend on the radius $r>0$ used to control the pressure coming from the outflow condition on $\Gamma_N$. As $\gamma_0$ is typically small, we can estimate
\[
\Big(1+\frac{\gamma_0}{2d}\Big)^{-1}\approx 1.
\]
To study the scaling with the size of the outflow boundary, we examine two typical configurations.:
\begin{itemize}
\item 
First, let $\Omega:=(0,L)\times (0,H)$ be a channel with $L>H$ and let $\Gamma_N=\{L\}\times (0,H)$ be the outflow boundary. Then, it holds ($d=2$, $r=H$, $c_1\approx H^2$, $c_2\approx H$)
\[
\gamma \approx \frac{H}{L} \min\left\{1,\frac{H^2}{H \sqrt{LH}}\right\}
= {\mathcal O}\left(\left(\frac{H}{L}\right)^\frac{3}{2}\right).
\]
Hence, the scaling is only slightly worse than that of the typical inf-sup constant. 
\item Second, we consider $d=3$ and a ball-shaped domain $\Omega$ with diameter $1$. 
We assume that the outflow boundary $\Gamma_N$ is a disc with diameter $r<1$. Then, we estimate
\[
\gamma \approx \gamma_0r^\frac{3}{2}.
\]
The degeneration $\gamma\to 0$ for $r\to0$ is physical, as a continuous transition from the Neumann problem to the Dirichlet problem does not hold in general, e.g. if the Dirichlet condition prescribes a net inflow
\[
\int_{\Gamma_D}\bu\cdot\vn\,ds >0. 
\]
\end{itemize}
What we have not considered in this note is the fixation of the pressure on the outflow boundary $\Gamma_N$, see~\eqref{pres-reg}. This imposes a relation between $p\in Q$ and $\bar p$ that might be utilized.


\end{document}